\documentclass{amsart}
\usepackage{enumerate}

\title[Largest Minimal Inversion/Pair-Complete Permutation Sets]{Largest Minimal Inversion-Complete and Pair-Complete Sets of Permutations}

\author{Eric Balandraud}
\address{
 \small Institut de Math\'{e}matiques de Jussieu,
 \small Universit\'{e} Pierre et Marie Curie (Paris 6),
 \small Paris, France
 }
\email{eric.balandraud@imj-prg.fr}
\author{Maurice Queyranne}
\address{
 \small Sauder School of Business,
 \small University of British Columbia,
 \small Vancouver, B.C., Canada
 }
\curraddr{
 \small Center for Operations Research and Econometrics (CORE),
 \small Universit\'e Catholique de Louvain,
 \small Louvain-la-Neuve, Belgium
 }
\email{Maurice.Queyranne@sauder.ubc.ca}
\author{Fabio Tardella}
\address{
 \small Dipartimento MEMOTEF,
 \small Sapienza University of Rome,
 \small Roma, Italy
 }
\email{fabio.tardella@uniroma1.it}

\date{\today}

\newtheorem{theorem}{Theorem}
\newtheorem{lemma}{Lemma}
\newtheorem{corollary}[lemma]{Corollary}
\newtheorem{remark}{Remark}
\newcommand{\Rset}{\mathbb R}
\newcommand{\Zset}{\mathbb Z}
\def\calF{\mathcal{F}}
\def\calP{\mathcal{P}}
\def\calQ{\mathcal{Q}}
\newcommand{\set}[1]{\left\{#1\right\}}
\newcommand{\Id}{\text{id}}
\newcommand{\Rev}{\text{rev}}

\begin{document}

\begin{abstract}
We solve two related extremal problems in the theory of permutations.
A set $Q$ of permutations of the integers 1 to $n$ is inversion-complete (resp., pair-complete)
if for every inversion $(j,i)$, where $1 \le i < j \le n$, (resp., for every pair $(i,j)$, where $i\not= j$)
there exists a permutation in~$Q$ where $j$ is before~$i$.
It is minimally inversion-complete if in addition no proper subset of~$Q$ is inversion-complete; and similarly for pair-completeness.
The problems we consider are to determine the maximum cardinality of a minimal inversion-complete set of permutations, and
that of a minimal pair-complete set of permutations.
The latter problem arises in the determination of the Carath\'eodory numbers for
certain abstract convexity structures on the $(n-1)$-dimensional real and integer vector spaces.
Using Mantel's Theorem on the maximum number of edges in a triangle-free graph,
we determine these two maximum cardinalities and
we present a complete description of the optimal sets of permutations for each problem.
Perhaps surprisingly (since there are twice as many pairs to cover as inversions), these two maximum cardinalities coincide whenever $n \ge 4$.
\end{abstract}

\maketitle

\newpage
We consider the following extremal problems in the theory of permutations.
Given integer $n\ge 2$, let
$S_n$
denote the symmetric group of
all permutations of $[n] := \set{1,2,\dots,n}$ (so $|S_n| = n!$),
and $A_n = \set{(i,j) : i,j\in [n],~i\not=j}$
the set of all (ordered) pairs from $[n]$ (so $|A_n| = n(n-1)$).
A permutation $\pi = (\pi(1),\dots,\pi(n))$ {\em covers}
the pair $(\pi(k),\pi(l))\in A_n$ iff $k < l$.
An \emph{inversion} (see, e.g., \cite{Bona,Margolius,Markowsky})
is a pair $(j,i)\in A_n$ with $j > i$.
Let $I_n\subset A_n$ denote the set of all inversions.
A set $Q\subseteq S_n$ of permutations is \emph{inversion-complete} (resp., \emph{pair-complete})
if every inversion in~$I_n$ (resp., pair in~$A_n$) is covered by at least one permutation in~$Q$.
An inversion-complete set $Q$ is \emph{minimally} inversion-complete if no proper subset of~$Q$ is inversion-complete;
and similarly for pair-completeness.
For example, the set $Q' = \set{\Rev_n}$, where (using compact notation for permutations) $\Rev_n = n(n-1)\dots 21$ is
the reverse permutation, is minimally inversion-complete, and has minimum cardinality for this property; whereas
the set $P' = \set{\Id_n, \Rev_n}$, where $\Id_n = 12\dots n$ is the identity permutation, is minimally pair-complete,
and has minimum cardinality.

We determine the \emph{maximum} cardinality
$\gamma_I(n)$
of a minimal inversion-complete subset $Q\subseteq S_n$,
as well as the maximum cardinality $\gamma_P(n)$,
of a minimal pair-complete subset $P\subseteq S_n$.
The latter problem arose in the determination of the Carath\'eodory numbers for
the integral $L^{\natural}$ convexity structures on the $(n-1)$-dimensional real and integer vector spaces
$\Rset^{n-1}$ and~$\Zset^{n-1}$, see \cite{QT}.\footnote{We refer the curious reader to
    van de Vel's monograph \cite{vdVel} for a general introduction to convexity structures and convexity invariants
    (such as the Carath\'eodory number), and to Murota's monograph \cite{Murota} on various models of discrete convexity,
    including $L^{\natural}$ and related convexities.}
It was posed by the second author as ``An Integer Programming Formulation Challenge'' at
the Integer Programming Workshop, Valparaiso, Chile, March 11-14, 2012.

Stimulated by personal communication of an early version of our results,
Malvenuto et al.~\cite{Malvenuto} determine the exact value of, or bounds on,
the maximum cardinality of minimal inversion-complete sets in
more general classes of finite reflection groups.

Perhaps unexpectedly (since
there are twice as many pairs to cover as inversions),
the maximum cardinalities $\gamma_I(n)$ and $\gamma_P(n)$ considered herein
are equal for all $n \ge 4$ (and they only differ by one unit, viz., $\gamma_P(n) = \gamma_I(n)+1$, for $n = 2$ and~3).
Furthermore, for all
$n\ge 4$
the family
$\calQ^*_n$
of all maximum-cardinality minimal inversion-complete subsets of~$S_n$
is \emph{strictly} contained in
the family $\calP^*_n$
of all maximum-cardinality minimal pair-complete subsets.
All our proofs are constructive and produce corresponding optimal sets of permutations.

\medbreak
In Section~\ref{sec:inversions} we prove:
\begin{theorem}\label{thm:inversions}
\begin{enumerate}[(i)]
  \item For every $n\ge 2$, the maximum cardinality
    of a minimal inversion-complete subset of~$S_n$
    is $\gamma_I(n) = \lfloor n^2 /4 \rfloor$.
  \item For every even $n\ge 4$, the family $\calQ^*_n$ of
    all maximum-cardinality minimal inversion-complete subsets of~$S_n$
    is the family of all transversals of a family of $n^2/4$ pairwise disjoint subsets of~$S_n$,
    each of cardinality $\left[\left(\frac{n}{2}-1\right) ! \right]^2$,
    and thus $|\calQ^*_n|
    = \left[\left(\frac{n}{2}-1\right) ! \right]^{n^2/2}$.
  \item For every odd $n\ge 5$, $\calQ^*_n$ is the disjoint union of the families
    of all transversals of two families, each one of $\lfloor n^2/4 \rfloor$ pairwise disjoint subsets of~$S_n$
    of cardinality $\left(\lfloor\frac{n}{2}\rfloor-1\right) !\; \lfloor\frac{n}{2}\rfloor !$,
    and thus $|\calQ^*_n|
= 2\left[\left(\lfloor\frac{n}{2}\rfloor-1\right) !\; \lfloor\frac{n}{2}\rfloor !\right]^{\lfloor n^2/4 \rfloor}$.
\end{enumerate}
\end{theorem}
\noindent
To prove Theorem~\ref{thm:inversions}, we first establish the upper bound $\gamma_I(n) \le \lfloor n^2 /4 \rfloor$ by
applying Mantel's Theorem
(which states, \cite{Mantel,Turan}, that the maximum number of edges in an $n$-vertex triangle-free graph is $\lfloor n^2/4\rfloor$)
to certain ``critical selection graphs'' associated with the minimal inversion-complete subsets of~$S_n$.
We then show that this upper bound is attained by the families of transversals described in parts \textit{(ii)--(iii)}.
We complete the proof by showing that, for $n\ge 4$, every $Q\in\calQ^*_n$ must be such a transversal.
Note that these results imply the asymptotic growth rate $|\calQ^*_n| = 2^{\theta(n^3 \log n)}$ as $n$ grows.

\par\medbreak
In Section~\ref{sec:pairs} we prove:
\begin{theorem}\label{thm:pairs}
\begin{enumerate}[(i)]
  \item For every integer $n\ge 2$, the maximum cardinality
    of a minimal pair-complete subset of~$S_n$
    is $\gamma_P(n) = \max\set{n, \lfloor n^2 /4 \rfloor}$.
  \item For all $n\ge 5$ the set
    $\calP^*_n$
    of maximum-cardinality minimal pair-complete subset of~$S_n$ is
    equal to the set
    $\tau\circ\calQ^*_n$
    resulting from applying every possible permutation $\tau\in S_n$ of the index set~$[n]$ to
    each $Q\in\calQ^*_n$.
  \item For all $n\ge 5$ there is a one-to-one correspondence between $\calP^*_n$
    and the Cartesian product $\binom{[n]}{\lfloor n/2\rfloor} \times \calQ^*_n$,
    where $\binom{[n]}{\lfloor n/2\rfloor}$ is the family of all subsets $S\subset [n]$ with cardinality~$|S| = \lfloor n/2\rfloor$.
\end{enumerate}
\end{theorem}
\noindent The intuition for the formula $\gamma_P(n) = \max\set{n, \lfloor n^2 /4 \rfloor}$ in part~\textit{(i)}
is that it suffices to consider two
classes
of minimal pair-complete subsets:
\begin{enumerate}[(1)]
\item the subsets $P$
    (each
    of cardinality~$n$) formed by the $n$ circular shifts of any given permutation $\pi\in S_n$,
    i.e., $P = \{\pi,\, \pi\circ\sigma,\, \pi\circ\sigma^2,\dots, \pi\circ\sigma^{n-1}\}$,
    where the (forward) circular shift $\sigma\in S_n$ is defined by $\sigma(i) = (i \mod n) + 1$ for all $i\in [n]$; and
\item the subsets
    $P = \tau\circ Q$ (each of cardinality~$\lfloor n^2 /4 \rfloor$)
    defined in part~\textit{(ii)} of Theorem~\ref{thm:pairs}.
\end{enumerate}
%
The characterization in part~\textit{(ii)} of Theorem~\ref{thm:pairs} only implies that $|\calP^*_n| \le n!\, |\calQ^*_n|$,
because different pairs $(\tau, Q)$ may give rise to the same set $\tau\circ Q$
(as will be seen, for example, in Remark~\ref{rem:pairs} at the end of this paper, with the ``class-(2) subsets'' for the case $n=4$ therein).
Part~\textit{(iii)}, on the other hand, refines the preceding result using a ``canonical'' permutation $\tau_W$
induced by a balanced partition $\{W,\overline{W}\}$ of the index set~$[n]$
(i.e., with $|W|$ or $|\overline{W}| = \lfloor n/2\rfloor$, a consequence of Mantel's Theorem).
This implies that $|\calP^*_n| = \binom{n}{\lfloor n/2\rfloor} |\calQ^*_n|$ for $n \ge 5$.
Thus, although $|\calP^*_n| > |\calQ^*_n|$ for all $n \ge 5$,
their asymptotic growth rate (as $n$ grows) are similar, differing only in lower order terms in the exponent ${\theta(n^3 \log n)}$.

\section{minimal Inversion-Complete Sets of Permutations}\label{sec:inversions}
In this Section we prove Theorem~\ref{thm:inversions} and present a characterization of the family
$\calQ^*_n$
of all maximum-cardinality minimal inversion-complete subsets of~$S_n$.
For $n=2$, there is a single inversion $(2,1)$, which is covered by the reverse permutation 21,
so part \textit{(i)} of Theorem ~\ref{thm:inversions} trivially holds and $\calQ^*_2 = \{21\}$.
Hence assume $n\ge 3$ in the rest of this Section.

Consider any minimal inversion-complete subset $Q$ of~$S_n$.
Since $Q$ is \emph{minimally} inversion-complete, for every permutation $\pi\in Q$ there exists an inversion $(j,i)\in I_n$,
called a \emph{critical inversion}, which is covered by $\pi$ and by no permutation in $Q\setminus\set{\pi}$
(for otherwise $Q\setminus\set{\pi}$ would also be inversion-complete, and thus $Q$ would not be minimally inversion-complete).
For every permutation $\pi\in Q$, select \emph{one} critical inversion that it covers
(arbitrarily chosen if $\pi$ covers more than one critical inversion).
Let $q_{j,i}$ denote the unique permutation in~$Q$ that covers the selected critical inversion $(j,i)$.
Consider a corresponding \emph{critical selection graph} $G_Q = ([n], E_Q)$, where $E_Q$ is the set of these $|Q|$ selected critical inversions
(one for each permutation in~$Q$), considered as undirected edges.
Thus $|E_Q| = |Q|$.

Recall that a graph $G$ is \emph{triangle-free} if there are no three distinct vertices $i$, $j$ and $k$ such that
all three edges $\set{i,j}$, $\set{i,k}$ and $\set{j,k}$ are in~$G$.

\begin{lemma}\label{lem:inv_triangle_free}
If subset $Q\subseteq S_n$ is minimally inversion-complete, then every corresponding critical selection graph $G_Q$ is triangle-free.
\end{lemma}
\begin{proof}
Assume $Q\subseteq S_n$ is minimally inversion-complete, and let $G_Q = ([n], E_Q)$ be a corresponding critical selection graph.
We need to show that, if $E_Q$ contains two adjacent edges $\set{i,j}$ and $\set{j,k}$, then it cannot contain edge $\set{i,k}$.
Thus assume that $\set{i,j}$ and $\set{j,k}\in E_Q$ and, without loss of generality, that $i < k$.
We want to show that $(k,i)$ cannot be a selected critical inversion.
We consider the possible relative positions of index $j$ relative to $i$ and $k$:
\begin{itemize}
\item If $j < i < k$, i.e., both $(i,j)$ and $(k,j)$ are selected critical inversions, then
  $q_{k,j}$ cannot cover $(i,j)$ and therefore
  we must have $k$ before $j$ before~$i$ in~$q_{k,j}$
  (that is, these three indices must be in positions $\pi^{-1}(i) < \pi^{-1}(j) < \pi^{-1}(k)$ in~$\pi = q_{k,j}$).
  This implies that $(k,i)$ cannot be a selected critical inversion.
\item If $i < k < j$, i.e., both $(j,i)$ and $(j,k)$ are selected critical inversions, then
  this is dual (in the order-theoretic sense) to the previous case:
  $q_{j,i}$ cannot cover $(j,k)$ and therefore we must have $k$ before $j$ before~$i$ in~$q_{j,i}$,
  implying that $(k,i)$ cannot be a selected critical inversion.
\item Else $i < j < k$, i.e., both $(j,i)$ and $(k,j)$ are selected critical inversions.
  In every permutation $\pi\in Q\setminus\set{q_{j,i}, q_{k,j}}$ we must have $i$ before $j$ before~$k$
But then $(k,i)$ cannot be a selected critical inversion, since it can only be covered in $Q$ by $q_{j,i}$ or $q_{k,j}$,
  for each of which another critical inversion has been selected.
\end{itemize}
Therefore, $(k,i)$ cannot be a selected critical inversion.
This implies that no three indices $i$, $j$ and~$k$ can define a triangle in~$G_Q$.
\end{proof}
Since $|Q| = |E_Q|$, Mantel's Theorem implies
\begin{corollary}\label{cor:inversions}
  For every $n\ge 2$, the maximum cardinality $\gamma_I(n)$ of
  a minimal inversion-complete subset of~$S_n$ satisfies $\gamma_I(n) \le \lfloor n^2 /4 \rfloor$.
\end{corollary}

We prove constructively that the upper bound in Corollary~\ref{cor:inversions} is attained, i.e., that
part \textit{(i)} of Theorem ~\ref{thm:inversions} holds.
For $n=3$, we have 3 triangle-free graphs on vertex set $\set{1,2,3}$,
each consisting of exactly two of the three possible edges.
Consider the edge set $E' = \big\{  \set{1,2}, \set{1,3} \big\}$:
if it is the edge set of a critical selection graph $G_{Q'}$,
then we must have $q'_{2,1} = 213\in Q'$ (for otherwise, $q'_{2,1}$ would also cover the inversion $(3,1)$,
contradicting that $(3,1)$ is also selected),
and similarly $q'_{3,1} = 312\in Q'$.  Thus $Q'$ must be the set $\set{213, 312}$,
which is indeed inversion-complete, and thus a largest minimal inversion-complete subset of~$S_3$.
This implies that $\gamma_I(n) = 3 = \lfloor \frac{n^2}{4} \rfloor$ holds for $n=3$.
Similarly, the  edge sets $E'' = \big\{  \set{1,2}, \set{2,3} \big\}$
and $E''' = \big\{  \set{1,3}, \set{2,3} \big\}$ define the other two maximum-cardinality
minimal inversion-complete subsets $Q'' = \set{213, 123}$ and $Q''' = \set{231, 321}$ of~$S_3$.
Thus
$\calQ^*_3  = \{Q',Q'',Q'''\}$ and $|\calQ^*_3|  = 3$.

Thus assume $n\ge 4$ in the rest of this Section.
We now introduce certain subsets of $S_n$, which we will use to show that
the upper bound in Corollary~\ref{cor:inversions} is attained,
and to construct
the whole set $\calQ^*$.
For every triple $(i,c,j)$ of integers such that $1 \le i \le c < j \le n$, let
$F_{i,c,j}$
denote the set of all permutations $\pi\in S_n$ such that:
\begin{itemize}
\item $\pi(h)\le c$ for all $h<c$;
\item $\pi(c)=j$;
\item $\pi(c+1)=i$; and
\item $\pi(k)\ge c+1$ for all  $k> c+1$.
\end{itemize}
If $c >1$ the first two conditions imply that $\big(\pi(1),\dots,\pi(c-1)\big)$ is any permutation of $[c]\setminus\{i\}$;
and if $c+1<n$ the last two conditions imply that $\big(\pi(c+2),\dots,\pi(n)\big)$ is any permutation of
$\{c+1,\dots,n\}\setminus\{j\}$.
Thus the cardinality of $F_{i,c,j}$ is $(c-1)!\,(n-c-1)!$.
Note also that, for every $\pi\in F_{i,c,j}$,
$(k,h) = (j,i)$ is the unique inversion $(k,h)$ with $h \le c < k$ that is covered by~$\pi$.
Thus for every fixed $c$ the sets $F_{i,c,j}$ ($1 \le i \le c < j \le n$) are pairwise disjoint
($F_{i,c,j}\cap F_{i',c,j'}=\emptyset$ whenever $(i,j)\neq(i',j')$).
Recall that, given a collection $\calF$ of sets,
a \emph{transversal} is a set containing exactly one element from each member of~$\calF$.
\begin{lemma}\label{lem:inv_transv_OK}
For every integers $1\le c < n$, every transversal $T$ of the family
$\calF_c = \{F_{i,c,j} : 1 \le i \le c < j \le n\}$ is minimally inversion-complete.
\end{lemma}
\begin{proof}
Given such a transversal~$T$, let $t_{i,j}$ denote the permutation in $T\cap F_{i,c,j}$.
For every inversion $(j,i)\in F_n$, we consider the relative positions of $i$ and $j$ with respect to $c$:
\begin{itemize}
\item If $i\le c<j$, then $t_{i,j}$ is the unique permutation in~$T$ that covers the inversion~$(j,i)$.
\item If $i<j\le c$, then the inversion $(j,i)$ is covered by every $t_{i,j'}\in T$ with $j'>c$.
\item Else, $c+1\le i<j$, then the inversion $(j,i)$ is covered by every $t_{i',j}\in T$ with $i'\le c$.
\end{itemize}
Therefore, $T$ is inversion-complete and
for every $i\le c<j$ the inversion $(j,i)$, covered by $t_{i,j}$, is critical.
This implies that $T$ is minimally inversion-complete.
\end{proof}
For a fixed $c$ such that $1\le c < n$, there are
$c(n-c)$ subsets $F_{i,c,j}$ (with $i\le c<j$) (and these subsets are nonempty and pairwise disjoint).
Hence the cardinality of every transversal $T$ satisfies
$|T| = |\calF_c| = c(n-c) \le \lfloor\frac{n^2}{4}\rfloor$,
with equality iff $c\in\big\{\lfloor\frac{n}{2}\rfloor,\lceil\frac{n}{2}\rceil\big\}$.
Combining with Lemma~\ref{cor:inversions}, we obtain:
\begin{corollary}\label{cor:inv_transv__opt}
  For every $n\ge 4$, $\gamma_I(n) = \lfloor n^2 /4 \rfloor$ and,
  for every $c\in\big\{\lfloor\frac{n}{2}\rfloor,\lceil\frac{n}{2}\rceil\big\}$,
  every transversal $T$ of the family $\calF_c = \{F_{i,c,j} : 1 \le i \le c < j \le n\}$
  is a maximum-cardinality minimal inversion-complete subset of~$S_n$.
\end{corollary}
Part \textit{(i)} of Theorem~\ref{thm:inversions} follows.
To prove
parts \textit{(ii)} and~\textit{(iii)},
we invoke the ``strong form'' of Mantel's Theorem \cite{Mantel,Turan}:
an $n$-vertex triangle-free graph has the maximum number $\lfloor n^2/4\rfloor$ of edges
iff it is a
\emph{balanced} bipartite graph,
i.e., with $\lfloor\frac{n}{2}\rfloor$ vertices on one side
and $\lceil\frac{n}{2}\rceil$ on the other.
\begin{lemma}\label{lem:inv_transv_all}
  For $n\ge 4$, a subset of~$S_n$ is a maximum-cardinality minimal inversion-complete subset iff
  it is a transversal of the family $\calF_c$ for some
  $c\in\big\{\lfloor\frac{n}{2}\rfloor,\lceil\frac{n}{2}\rceil\big\}$.
\end{lemma}
\begin{proof}
Sufficiency was established by Lemma~\ref{lem:inv_transv_OK}.
To prove necessity, let $n\ge 4$ and consider
any $Q\in\calQ^*_n$
and a corresponding critical selection graph $G_Q$.
By Lemma~\ref{lem:inv_triangle_free}, Corollary~\ref{cor:inv_transv__opt}, and the strong form of Mantel's Theorem,
$G_Q$ is a
balanced complete bipartite graph.
We first claim that the side $W$ of~$G_Q$ that contains index~$1$ must be $W = [c]$
with $c\in\big\{\lfloor\frac{n}{2}\rfloor,\lceil\frac{n}{2}\rceil\big\}$.
For this, consider (since $n\ge 4$) any three indices $i > 1$ in $W$ and $j < k$ on the other side.
Thus $(j,1)$ and $(k,1)$ are critical inversions.
Furthermore, the edges $\{i,j\}$ and $\{i,k\}$ in~$G_Q$ are also defined by critical inversions,
which depend on the position of index~$i$ relative to $j$ and~$k$:
\begin{itemize}
\item If $1 < j < i < k$, then $(i,j)$ and $(k,i)$ are critical inversions.
  Then every permutation $\pi\in Q\setminus\set{q_{j,1}, q_{i,j}, q_{k,i}}$ has
  $1$ before $j$ before $i$ before $k$, and thus does not cover the inversion $(k,1)$.
  Thus $(k,1)$ cannot be a critical inversion, a contradiction.
\item If $1 < j < k < i$, then $(i,j)$ and $(i,k)$ are critical inversions.
  On one hand, every permutation $\pi\in Q' = Q\setminus\set{q_{j,1}, q_{i,j}}$ has
  $1$ before $j$ before $i$, and thus does not cover the inversion $(i,1)$.
  Similarly, every permutation $\pi\in Q'' = Q\setminus\set{q_{k,1}, q_{i,k}}$ has
  $1$ before $k$ before $i$, and thus does not cover the inversion $(i,1)$ either.
  Therefore $Q = Q' \cup Q''$ does not cover the inversion $(i,1)$, a contradiction.
\end{itemize}
This implies that we must have $1 < i < j < k$, i.e., that $i < j$ for every $i\in W$ and every $j\in [n]\setminus W$.
This proves our claim that $W = [c]$ for some $c$ which, by the strong form of Mantel's Theorem,
must be $\lfloor\frac{n}{2}\rfloor$ or $\lceil\frac{n}{2}\rceil$.

As a consequence, $Q = \{q_{j,i} : 1 \le i \le c < j \le n\}$.
Every $q_{j,i}\in Q$ must have $j$ before $i$, and also $h$ before $j$ for every $h\in[c]\setminus\{i\}$
(for otherwise $q_{j,i}$ would also cover the inversion $(j,h)$, contradicting the fact that
$q_{j,h}$ is the unique permutation in~$Q$ that covers~$(j,h)$) and
$i$ before $k$ for every $k\in\{c+1,\dots,n\}\setminus\{j\}$
(for otherwise $q_{j,i}$ would also cover the inversion $(k,i)$).
Therefore $q_{j,i}\in F_{i,c,j}$, and thus $Q$ is a transversal of $\calF_c$.
The proof is complete.
\end{proof}

Parts \emph{(ii)} and~\textit{(iii)} of Theorem~\ref{thm:inversions} now follow, noting that:
(1) $\calF_c$ consists of $c(n-c)$ pairwise disjoint subsets $F_{i,c,j}$;
(2) $c\in\big\{\lfloor\frac{n}{2}\rfloor,\lceil\frac{n}{2}\rceil\big\}$; and
(3) for $n$ odd, all subsets $F_{i, \lfloor n/2\rfloor, j}$ and $F_{i', \lceil n/2\rceil, j'}$
are pairwise disjoint (indeed, with $n$ odd,
every $\pi\in F_{i, \lfloor n/2\rfloor, j}$ has $\pi(c+1)\in [c]$ while
every $\pi\in F_{i', \lceil n/2\rceil, j'}$ has $\pi(c+1)\in [n]\setminus [c]$).

\begin{remark}
Thus we have
$\calQ^*_2 = 1$, $\calQ^*_3 = 3$, $\calQ^*_4 = 1$, $\calQ^*_5 = 128$ and, as noted in the Introduction,
the asymptotic growth rate $|\calQ^*_n| = 2^{\theta(n^3 \log n)}$.
\end{remark}

\section{minimal Pair-Complete Sets of Permutations}\label{sec:pairs}
In this Section we prove Theorem~\ref{thm:pairs}.
To simplify the presentation, let $\mu(n) := \max\set{n, \lfloor n^2 /4 \rfloor}$.
For $n=2$, the unique cover of the two pairs $(1,2)$ and $(2,1)$
is~$S_2$ itself, hence $\gamma_P(2) = 2 = \mu(2)$
%
and $\calP^*_2 = \{S_2\}$.

\par\medskip\nobreak
Note that, as for inversions, given any minimal pair-complete subset $P$ of~$S_n$, for every permutation $\pi\in P$
there exists a \emph{critical pair} $(i,j)\in F_n$ which is covered by $\pi$ and by no other permutation in $P$.
Observe however that, in contrast with inversion-completeness, the notion of pair-com\-plete\-ness does not assume
any particular order of the indices.
Thus, if $P\subseteq S_n$ is (minimally) pair-complete then for any permutation $\tau\in S_n$ of the index set~$[n]$,
the set
$\tau\circ P = \{\tau\circ\pi :
\pi\in P\}$ is also (minimally) pair-complete.
(Indeed, $\pi$ covers $(i,j)$ iff
$\tau\circ\pi$
covers $(\tau(i),\tau(j))$.)

For $n=3$ consider the set $P_3 := \set{123, 231, 312}$.
It is easily verified that $P_3$ is pair-complete and
the pairs $(1,3)$, $(2,1)$ and $(3,2)$ are critical pairs covered by the permutations 123, 231 and 312, respectively.
Hence
$P_3\in\calP^*_3$ and thus
$\gamma_P(3) \ge |P_3| = 3 = \mu(3)$.
To verify the converse inequality, viz., $\gamma_P(3) \le \mu(3)$, consider any
$P\in\calP^*_3$:
by the preceding observation, we may assume, w.l.o.g., that
$P$ contains the identity permutation $\pi_1 = \Id_3$.
This permutation $\pi_1$ covers all three pairs $(i,j)$ with $i<j$.
Then the permutation $\pi_2\in P$ that covers the pair $(3,1)$ must also (depending of the position of index~2)
cover at least one of the pairs $(2,1)$ or $(3,2)$.
Hence there is at most one pair which is not covered by $\set{\pi_1,\pi_2}$, and thus $\gamma_P(3) = |P| \le 3 = \mu(3)$,
implying $\gamma_P(3) = \mu(3)$.
Therefore part \textit{(i)} of Theorem ~\ref{thm:pairs} holds for $n\in\set{2,3}$.

\begin{lemma}\label{lem:pair_transv_OK}
If $n\ge 4$, for every permutation $\tau\in S_n$ of the index set $[n]$ and
every maximum-cardinality minimal inversion-complete set $Q\subset S_n$,
the set
$\tau\circ Q$
is minimally pair-complete.
\end{lemma}
\begin{proof}
By a preceding observation, it suffices to prove that, for $n\ge 4$,
every maximum-cardinality minimal inversion-complete set $Q\in S_n$ is minimally pair-complete.
By Lemma~\ref{lem:inv_transv_all}, every such $Q$ must be
a transversal of~$\calF_c$ for some $c\in\big\{\lfloor\frac{n}{2}\rfloor,\lceil\frac{n}{2}\rceil\big\}$.
Consider any pair $(i,j)\in A_n$ and, w.l.o.g., $i<j$:
\begin{itemize}
\item If $i\le c < j$ then $q_{j,i}\in F_{i,c,j}\cap Q$ is the unique permutation in~$Q$ that covers the inversion $(j,i)$,
    and every other permutation in~$Q$ covers~$(i,j)$.
\item If $i < j \le c$ then, for every $k\in \{c+1,\dots,n\}$, $q_{k,i}$ covers $(j,i)$ and $q_{k,j}$ covers $(i,j)$.
\item Else, $c < i < j$ and, dually, for every $h\in [c]$, $q_{i,h}$ covers $(i,j)$ and $q_{j,h}$ covers $(j,i)$.
\end{itemize}
Therefore $Q$ is pair-complete, and every pair $(j,i)$ with $i\le c < j$ is critical
and covered by $q_{j,i}\in Q$.
Since $|Q| = \lfloor\frac{n^2}{4}\rfloor = |\{(j,i) : 1\le i\le c < j \le n\}|$,
$Q$ is minimally pair-complete.
\end{proof}

\begin{corollary}\label{cor:pair_LB}
  For every $n\ge 4$, the maximum cardinality $\gamma_P(n)$ of
  a minimal pair-complete subset of~$S_n$ satisfies $\gamma_P(n) \ge \gamma_I(n) = \lfloor n^2 /4 \rfloor$.
\end{corollary}

As we did for inversions, to every minimal pair-complete subset $P$ of~$S_n$ and
selection of a critical pair covered by each permutation in~$P$,
we associate a corresponding \emph{critical selection graph} $G_P = ([n], E_P)$ where
$E_P$ is the set of $|P|$ selected critical pairs (one for each permutation in~$P$), considered as undirected edges.
Thus $|E_P| = |P|$.
Let $p_{i,j}$ denote the unique permutation in~$P$ that covers the selected critical pair $(i,j)$.
\begin{lemma}\label{lem:pair_triangle_free}
If $n\ge 4$ and $P\subseteq S_n$ is minimally pair-complete, then every corresponding critical selection graph $G_P$ is triangle-free.
\end{lemma}
\begin{proof}
Assume $n\ge 4$ and $P\subseteq S_n$ is minimally pair-complete, and let $G_P = ([n], E_P)$ be a corresponding critical selection graph.
By Corollary~\ref{cor:pair_LB}, $|P| \ge \lfloor n^2 /4 \rfloor \ge 4$.
We have to show that for any three indices $i,j,k$ such that $\set{i,j}$ and $\set{j,k}\in E_P$,
we must have $\set{i,k}\not\in E_P$.
\begin{itemize}
\item First, consider the case where both pairs $(i,j)$ and $(j,k)$ are critical.
  In every permutation $\pi\in P\setminus\set{p_{i,j}, p_{j,k}}$ we must thus have $k$ before $j$ before~$i$.
  Since $k$ is before $i$ in all these $|P| -2 \ge 2$ permutations, $(k,i)$ cannot be a critical pair.
  Furthermore $(i,k)$ cannot be a selected critical pair, since it can only be covered by $p_{i,j}$ and $p_{j,k}$,
  for each of which another critical pair has been selected.
  Therefore, as claimed, we cannot have $\set{i,k}$ in $E_P$.
\item A dual argument shows that if both $(j,i)$ and $(k,j)$ are selected critical pairs then $\set{i,k}\not\in E_P$.
\item Now consider the case where $(j,i)$ and $(j,k)$ are selected critical pairs.
  Since $p_{j,i}$ does not cover $(j,k)$, we have $k$ before $j$ before $i$ in $p_{j,i}$, implying that $(k,i)$ cannot be a selected critical pair.
  Similarly, $p_{j,k}$ does not cover $(j,i)$ and therefore
  we must have $i$ before $j$ before $k$ in $p_{j,k}$, implying that $(i,k)$ cannot be a selected critical pair.
  Therefore, as claimed, we cannot have $\set{i,k}$ in $E_P$.
\item A dual argument applies to the remaining case, showing that if both $(i,j)$ and $(k,j)$ are selected critical pairs then $\set{i,k}\not\in E_P$.
\end{itemize}
Thus we must have $\set{i,k}\not\in E_P$.
This completes the proof that $G_P$ is triangle-free.
\end{proof}
These results and Mantel's Theorem imply part \textit{(i)} of Theorem~\ref{thm:pairs}.
They also imply that, for $n\ge 4$, all the sets
$\tau\circ Q$
in Lemma~\ref{lem:pair_transv_OK} are
in~$\calP^*_n$.
To complete the proof of
part~\textit{(ii)}
it now suffices to prove the converse for $n \ge 5$.
\begin{lemma}\label{lem:pair_transv_all}
If $n\ge 5$, a subset $P$ of~$S_n$ is a maximum-cardinality minimal pair-complete subset iff
$P = \tau\circ Q$
for some permutation $\tau$ of the index set $[n]$ and
some maximum-cardinality minimal inversion-complete subset $Q$ of~$S_n$.
\end{lemma}
\begin{proof}
Sufficiency was just established.
To prove necessity, let $n\ge 5$ and consider
any $P\in\calP^*_n$
and a corresponding critical selection graph $G_P$.
By Lemma~\ref{lem:pair_triangle_free}, Theorem~\ref{thm:pairs}~\textit{(i)}, and the strong form of Mantel's Theorem,
$G_P$ is a
balanced complete bipartite graph.

Now, also consider the associated  critical selection \emph{digraph} (directed graph) $D_P = ([n], A_P)$,
wherein each edge $\set{i,j}$ is directed as arc $(i,j)\in A_P$ if the pair $(i,j)$ is critical.
If $(i,j)\in A_P$, then the reverse pair $(j,i)$ must be covered by every permutation in $P\setminus\set{p_{i,j}}$;
since $|P| -1 = \lfloor n^2 /4 \rfloor - 1 \ge 2$ when $n \ge 5$,
$(j,i)$ cannot be critical, i.e., $(j,i)\not\in A_P$.
Thus, every edge $\set{i,j}$ in the underlying graph $G_P$ of~$D_P$ corresponds to exactly one arc,
$(i,j)$ or $(j,i)$, in~$D_P$.
We now prove that when $n \ge 5$ the digraph $D_P$ is acyclic, i.e., it does not contain any (directed) circuit.
Indeed, if $D_P$ contains a directed path $(i(1),\dots,i(k))$, then
every permutation $\pi\in Q\setminus\{p_{i(1),i(2)},\, p_{i(2),i(3)}\dots,\, p_{i(k-1),i(k)}\}$ has
$i(1)$ after $i(2)$ after $i(3)$, etc, after $i(k-1)$ after $i(k)$, and thus $i(k)$ before $i(1)$.
Since $|P| - (k-1) \ge \lfloor\frac{n^2}{4}\rfloor - (n-1) \ge 2$ when $n\ge 5$, pair $(i(k),i(1))$ cannot be critical,
and thus $(i(k),i(1))\not\in A_P$.

Since digraph $D_P$ is acyclic, there is at least one vertex $i$ with in-degree zero.
Since the underlying graph $G_P$ is a
balanced complete bipartite graph,
the side $W$ of~$G_Q$ that contains index~$i$ has cardinality
$|W| \in\big\{\lfloor\frac{n}{2}\rfloor,\,   \lceil\frac{n}{2}\rceil\big\}$.
Consider any indices $j\in W$ and $k \not= l$ on the other side.
Since vertex~$i$ has no entering arc, arcs $(i,k)$ and $(i,l)$ are in~$A_P$, and
we consider the possible orientations of the edges $\{j,k\}$ and $\{j,l\}$:
\begin{itemize}
\item If both edges are oriented into~$j$, i.e., $(k,j)$ and $(l,j)$ in~$A_P$, then, on one hand,
  every permutation $\pi\in P' = P\setminus\set{p_{k,j},\, p_{i,k}}$ has
  $j$ before $k$ before~$i$, and thus does not cover the pair $(i,j)$.
  Similarly, every permutation $\pi\in P'' = P\setminus\set{p_{l,j},\, p_{i,l}}$ has
  $j$ before $l$ before~$i$, and thus does not cover the pair $(i,j)$ either.
  Therefore $P = P' \cup P''$ does not cover the pair $(i,j)$, a contradiction.
\item If one of these two edges is oriented into~$j$ and the other one from~$j$,
  w.l.o.g., $(k,j)$ and $(j,l)$ in~$A_P$, then every permutation
  $\pi\in P\setminus\set{p_{j,l},\, p_{k,j},\, p_{i,k}}$ has
  $l$ before $k$ before~$i$, and thus does not cover the pair $(i,l)$.
  Thus $(i,l)$ cannot be a critical inversion, a contradiction.
\end{itemize}
Thus we must have both $(j,k)$ and $(j,l)$ in~$A_P$.
This implies that all pairs $(j,k)$ with $j\in W$ and
$k\in \overline{W} := [n]\setminus W$
define arcs in~$A_P$, i.e., are critical.
Let $c := |\overline{W}|\in\big\{\lfloor\frac{n}{2}\rfloor,\, \lceil\frac{n}{2}\rceil\big\}$
and consider any permutation $\tau$ that sends
$[c]$ to~$\overline{W}$ (and thus $\overline{[c]}$ to~$W$):
every critical pair in
$Q = \tau^{-1}\circ P$
is an inversion, hence $Q$ is minimally inversion-complete.
Since $|Q| = |P| = \gamma_P(n) = \gamma_I(n)$, it has maximum cardinality.
This completes the proof.
\end{proof}

It remains to prove:
\begin{lemma}\label{lem:bijection}
\emph{(Part \textit{(iii)} of Theorem~\ref{thm:pairs}.)}
For all $n\ge 5$ there is a one-to-one correspondence between $\calP^*_n$
and the Cartesian product $\binom{[n]}{\lfloor n/2\rfloor} \times \calQ^*_n$.
\end{lemma}
\begin{proof}
Assume $n\ge 5$ and consider any $P\in\calP^*_n$.
In the proof of Lemma~\ref{lem:pair_transv_all} we showed that there exists a unique balanced ordered partition
$(W, \overline{W})$ of~$[n]$ such that
all critical pairs $(j,k)$ of~$P$ have $j\in W$ and $k\in \overline{W}$.
Let again $c := |\overline{W}|\in\big\{\lfloor\frac{n}{2}\rfloor,\, \lceil\frac{n}{2}\rceil\big\}$,
but now select the (unique) ``canonical'' permutation $\tau_W$ associated with this ordered partition,
that \emph{monotonically} maps $[c]$ to~$\overline{W}$ (i.e., such that
$1\le i < j\le c$ implies $\tau_W(i)\in \overline{W}$ and $\tau_W(i) < \tau_W(j)\in \overline{W}$)
and monotonically maps $\overline{[c]} = \{c+1,\dots,n\}$ to~$W$ (i.e., such that
$c+1\le k < l\le n$ implies $\tau_W(k)\in W$ and $\tau_W(k) < \tau_W(l)\in W$).
Then, as also noted in the proof of Lemma~\ref{lem:pair_transv_all},
every critical pair in $Q_P := \tau_W^{-1}\circ P$ is an inversion,
and thus $Q_P \in \calQ^*_n$.

Consider the mapping $\Phi : \calP^*_n \mapsto \binom{[n]}{\lfloor n/2\rfloor} \times \calQ^*_n$
defined by $\Phi(P) = \big(\Phi(P)_1, \Phi(P)_2\big)$
with $\Phi(P)_1 = \overline{W}$ if $Q_P$ is a transversal of $\calF_{\lfloor n/2\rfloor}$
(i.e., if $|\overline{W}| = \lfloor n/2\rfloor$, where
$\{W,\overline{W}\}$ is the balanced ordered partition associated with~$P$, as defined in the preceding paragraph), and
$\Phi(P)_1 = W$ otherwise (i.e., if $n$ is odd and $Q_P$ is a transversal of $\calF_{\lceil n/2 \rceil}$, and
thus $|W| = \lfloor n/2\rfloor$);
and with $\Phi(P)_2 = Q_P = \tau_W^{-1}\circ P$.
To complete the proof, it suffices to show that every pair $(X,Q)\in \binom{[n]}{\lfloor n/2\rfloor} \times \calQ^*_n$
is the image $\Phi(P)$ of exactly one $P\in\calP^*_n$.

Thus consider any $(X,Q)\in \binom{[n]}{\lfloor n/2\rfloor} \times \calQ^*_n$.
Recall that the balanced ordered partition associated with any transversal $Q$ of~$\calF_c$ is $(\overline{[c]},[c])$.

If $Q$ is a transversal of $\calF_{\lfloor n/2\rfloor}$ then we use $\overline{X}$ to play the role of~$W$, that is,
we let $P := \tau_{\overline{X}}\circ Q$, so $P\in\calP^*_n$ and $\Phi(P)_2 = \tau_{\overline{X}}^{-1}\circ P = Q$.
The balanced ordered partition associated with~$P$ is
$(\tau_{\overline{X}}\circ \overline{[c]},\;\tau_{\overline{X}}\circ [c])  = (\overline{X}, X)$, and
therefore $\Phi(P)_1 = X$.
This implies that $\Phi(P) = (X,Q)$, as desired.
Furthermore, consider any $P'\in\calP^*_n$ such that $\Phi(P') = (X,Q)$.
Since $Q_{P'} = \Phi(P')_2 = Q$ is a transversal of $\calF_{\lfloor n/2\rfloor}$ and $\Phi(P')_1 = X$,
the balanced ordered partition associated with~$P'$ is $(W', \overline{W'}) = (\overline{X}, X)$.
But then $P' = \tau_{W'}\circ Q_{P'}= \tau_{\overline{X}}\circ Q = P$.
Therefore, for every $(X,Q)\in \binom{[n]}{\lfloor n/2\rfloor} \times \calQ^*_n$ such that
$Q$ is a transversal of $\calF_{\lfloor n/2\rfloor}$, there exists exactly one $P\in\calP^*_n$
such that $\Phi(P) = (X,Q)$.

The proof for the remaining case, i.e., when $n$ is odd and $Q$ is a transversal of $\calF_{\lceil n/2\rceil}$,
is similar, by simply exchanging the roles of $X$ and $\overline{X}$.
This shows that $\Phi$ is a one-to-one correspondence from $\calP^*_n$
to $\binom{[n]}{\lfloor n/2\rfloor} \times \calQ^*_n$.
\end{proof}

The proof of Theorem~\ref{thm:pairs} is complete

\begin{remark}\label{rem:pairs}
As seen at the beginning of Section~\ref{sec:pairs}, $\calP^*_2 = 1 =\calQ^*_2$.
For $n=3$ it can be verified that
$\calP^*_3$ consists of the two
orbits $\{123,\ 312,$ $231\}$ and $\{132,\, 213,\ 321\}$ of the circular shift,
thus $|\calP^*_3| = 2$ (while $|\calQ^*_3| = 3$).

\par\smallbreak
For $n=4$ we have two classes of maximum-cardinality minimal pair-complete subsets
(mentioned in the introduction):
\begin{enumerate}[(1)]
\item the $3! = 6$ orbits
  $P = \{\pi,\, \pi\circ\sigma,\, \pi\circ\sigma^2,\dots, \pi\circ\sigma^{n-1}\}$
  of the circular shift~$\sigma$, one for each permutation $\pi = \rho 4$ (permutation $\rho$ followed by 4) defined by each $\rho\in S_3$; and
\item the $\binom{4}{2} = 6$ distinct sets
  $P = \tau\circ Q$
  where $Q$ is the (unique) maximum-cardinality minimal pair-complete subset of~$S_4$, namely,
  the sets $P_{i,j} = \{ijkl,\, ilkj,\, kjil,$ $klij\}$ where $1\le i<j\le 4$ and
  $\{k,l\} = [4]\setminus\{i,j\}$.
  \end{enumerate}
Thus $|\calP^*_4| = 12$
(while $|\calQ^*_4| = 1$).

\par\smallbreak
For $n\ge 5$, part~\textit{(iii)} of Theorem~\ref{thm:pairs} implies that
$|\calP^*_n| = \binom{n}{\lfloor n/2\rfloor} |\calQ^*_n|$.
Thus, for example, $|\calP^*_5| = 10\,|\calQ^*_5| = 128$, and so on,
with the same asymptotic growth rate $|\calP^*_n| = 2^{\theta(n^3 \log n)}$ as $|\calQ^*_n|$.
\end{remark}

\end{document}